\documentclass[10pt]{article}

\usepackage{tikz}
\usepackage{tikz-cd}
\usetikzlibrary{matrix, arrows,  decorations.markings,  patterns,  plotmarks}
\usepackage{amsmath}
\usepackage{amssymb}
\usepackage{mathtools}
\usepackage{amsthm}

\usepackage[backend=bibtex, style=alphabetic, maxcitenames=5, maxbibnames=9 ]{biblatex}
\renewbibmacro{in:}{} 
\usepackage{subcaption}

\usepackage{hyperref}
\hypersetup{%
  bookmarksnumbered=true,%
  colorlinks=true,%
  linkcolor=blue,%
  citecolor=blue,%
  filecolor=blue,%
  menucolor=blue,%
  urlcolor=blue,%
  pdfnewwindow=true,%
  pdfstartview=FitBH}
\usepackage{cleveref}

\usepackage{todonotes}

\newcounter{mainthm}
\newtheorem{mainthm}[mainthm]{Theorem}

\newtheorem{mainprop}[mainthm]{Proposition}

\newtheorem{thm}{Theorem}[section]
\newtheorem{lem}[thm]{Lemma}

\newtheorem{prop}[thm]{Proposition}
\newtheorem{cor}[thm]{Corollary}
\newtheorem{corollary}[thm]{Corollary}

\theoremstyle{definition}
\newtheorem{defn}[thm]{Definition}
\newtheorem{definition}[thm]{Definition}

\theoremstyle{remark}
\newtheorem{rmk}[thm]{Remark}
\newtheorem{remark}[thm]{Remark}

\newtheorem{example}[thm]{Example}

\crefname{claim}{claim}{claims}
\crefname{prop}{proposition}{propositions}
\crefname{prop}{proposition}{propositions}
\crefname{assumption}{assumption}{assumptions}
\include{style/mathspreamble.tex}

\def\C{\mathbb{C}}

\def\N{\mathbb{N}}
\def\P{\mathbb{P}}

\def\R{\mathbb{R}}

\def\Z{\mathbb{Z}}

\DeclareMathOperator{\Aff}{Aff}

\DeclareMathOperator{\Alb}{Alb}

\DeclareMathOperator{\alb}{alb}
\DeclareMathOperator{\Cob}{Cob}

\DeclareMathOperator{\CH}{CH}

\DeclareMathOperator{\deff}{def}

\DeclareMathOperator{\GL}{GL}
\DeclareMathOperator{\Hom}{Hom}

\DeclareMathOperator{\Jac}{Jac}
\DeclareMathOperator{\Pic}{Pic}

\newcommand{\into}{\hookrightarrow}
\def\inv{^{-1}}
\newcommand{\mcal}{\mathcal}

\newcommand{\restr}[1]{|_{#1}}
\newcommand{\set}[1]{\left\lbrace{#1}\right\rbrace}
\newcommand{\st}{\,\vert\,}
\newcommand{\tx}[1]{\text{#1}}

\newcommand{\vspan}[1]{\langle #1 \rangle}
\newcommand{\wdg}{\wedge}
\def\x{\times}
\newcommand{\xra}[1]{\xrightarrow{#1}}

\title{Dimensionality of tropical Chow groups}
\author{Álvaro Muñiz-Brea}
\date{}

\begin{document}

\maketitle

\begin{abstract}
    We show that the existence of non-zero tropical forms of degree at least two implies that the tropical Chow group of points of an integral affine manifold is infinite-dimensional. 
    This can be seen as a tropical analog of classical results of Mumford and Roitman for Chow groups of smooth (complex) projective algebraic varieties. 
    We also show that the existence of tropical 1-forms on integral affine surfaces does not imply infinite dimensionality by considering the case of a tropical Klein bottle.
\end{abstract}

\tableofcontents

\section{Introduction}
\subsection{Algebraic Chow groups}\label{sec:Chowgroupsintro}
A lot of information about the topology of a manifold can be understood via its homology.
Similarly, in algebraic geometry, one can try to understand a variety by studying its subvarieties modulo some equivalence relation.
One of the most common equivalence relations is \emph{rational equivalence}: $k$-dimensional subvarieties $V_\pm \subset Y$ of an algebraic variety $Y$ are said to be \emph{rationally equivalent} if there exists a $(k+1)$-dimensional subvariety $W \subset Y \x \P^1$ whose projection to $\P^1$ is dominant and such that\footnote{We denote by $\pm\infty \in \P^1$ the north and south pole of $\P^1 \cong S^2$. Any two points of $\P^1$ could be used, however.}
$$
V_\pm = W \cap (Y \x \set{\pm\infty}).
$$
One defines the \emph{Chow groups of $Y$} as the quotient
\begin{equation}\label{eq:Chowgroups}
    \CH_k(Y) := Z_k(X)/R_k(X)
\end{equation}
where $Z_k(X)$ is the free abelian group generated by the $k$-dimensional subvarieties of $Y$ and $R_k(X) \subset Z_k(X)$ is the subgroup generated by differences $V_+ - V_-$ for all rationally equivalent pairs $V_-,V_+$.

Chow groups are classical objects in algebraic geometry and have been extensively studied.
They are the subject of the Hodge conjecture and a series of conjectures due to Bloch and Beilinson.
When $Y$ is a curve the Abel-Jacobi theorem---recall it identifies the Jacobian of a curve with its Albanese variety--- proves these groups are isomorphic to $H_{2*}(Y) \oplus Jac(Y)$.
However, as soon as we increase the dimension these groups become notoriously difficult to compute.
One of the first results along these lines is a classical theorem of Mumford in the '60s \cite{mumford1969rational}.
Mumford showed that when a smooth complex projective surface $Y$ admits a non-zero $2$-form---that is, when $0 < \dim H^0(\mcal K_Y)$, where $\mcal K_Y$ is the canonical bundle of $Y$---the Chow group of points $\CH_0(Y)$ is \emph{infinite-dimensional} in the following sense:
\begin{defn}[\cite{mumford1969rational}]\label{def:infinitedimensional}
    Let $\CH_0(Y)_{\hom} \subset \CH_0(Y)$ be the subgroup of cycles of degree zero (that is, those linear combinations $\sum n_i p_i$ with $\sum n_i = 0$).
    We write $\mathbf{y} = (y_1,\dots, y_k) \in Y^k$ to denote points in the $k$-th product of $Y$ with itself.
    We say $\CH_0(Y)_{\hom}$ is \emph{infinite-dimensional} if the maps
    \begin{align*}
        Y^k \x Y^k &\to \CH_0(Y)_{\hom}\\
        (\mathbf{y}^+,\mathbf{y}^-)& \mapsto \sum_{i = 1}^k\left(y_i^+ - y_i^-\right)
    \end{align*}
    are not surjective for any $k$. 
\end{defn}
Mumford's result was later generalized by Roitman to higher dimensions \cite{roitman1971gamma}, showing that for any smooth projective algebraic variety $Y$ over an uncountable field of characteristic zero,  the existence of a non-zero $p$-form for some $p\geq2$ is enough to guarantee infinite-dimensionality of $\CH_0(Y)_{\hom}$.

The converse result---that is, whether the vanishing $H^0(\mcal K_Y) = 0$ implies that $\CH_0(Y)_{\hom}$ is finite-dimensional---has been conjectured by Bloch \cite{bloch1975k2}, and Bloch-Kas-Lieberman \cite{bloch1976zero} showed that this is true for surfaces with Kodaira dimension $\kappa(Y) < 2$.

 \subsection{Tropical Chow groups}
An analogous equivalence relation for tropical geometry has appeared in the literature before, see e.g. \cite[Definition 3.1]{allermann2016rational}.
Let us recall the equivalence relation therein (with a slight change of notation and nomenclature for convenience).
In the following definition, given a tropical cycle $C$ and a rational function $\phi: C \to \R$, we denote by $\phi \cdot C$ the \emph{divisor of $\phi$}.\footnote{A \emph{rational function} on a tropical cycle is, roughly speaking, a piecewise-linear function with integer slope. Its divisor $\phi \cdot C$ is the locus where $\phi$ changes slope. We refer to the same paper \cite{allermann2016rational} for precise definitions.}
\begin{definition}[\cite{allermann2016rational}]\label{def:AllermanRationalEquivalence}
    Let $B$ be a tropical cycle and consider a $k$-dimensional subcycle $V \subset B$.
    We say that  $V$ is \emph{(bounded) rationally equivalent to zero} on $B$ if there exists a $(k+1)$-dimensional tropical cycle $f:C^{k+1} \to B$ and a (bounded) rational function $\phi : C \to \R$ such that 
    $$
    V = f_*(\phi \cdot C).
    $$
    We say that $V_+ \sim V_-$ are \emph{rationally equivalent} if $V_+ - V_-$ is rationally equivalent to zero.
\end{definition}

The setting in this paper is much more restrictive. 
We limit our attention and results to a subclass of tropical cycles $B$ called \emph{tropical affine manifolds}: these are smooth manifolds together with an integrable lattice $T_\Z B \subset TB$ of tangent vectors.\footnote{An equivalent definition is a smooth manifold together with an atlas of charts whose transition functions belong to the group $\Aff_\Z(\R^n)$ of integral affine transformations of $\R^n$, i.e. those of the form $\mathbf{x} \mapsto A\mathbf{x} + \mathbf{x}_0, A \in \GL(n,\Z)$.}
Using \Cref{def:AllermanRationalEquivalence}, one can define the \emph{tropical Chow groups} just as in \Cref{eq:Chowgroups}, except now  $Z_k(B)$ is generated by $k$-dimensional \emph{tropical subvarieties} of $B$ and $R_k(B) \subset Z_k(B)$ by those differences $V_+ - V_- \in Z_k(B)$ such that $V_+ \sim V_-$.
Since we will only be working with the tropical Chow group of points (that is,  the case $k = 0$), it is sufficient to say that a $1$-dimensional tropical subvariety---also known as a \emph{tropical curve}---is a weighted, balanced rational polyhedral complex of pure dimension one \cite{mikhalkin2005enumerative} (see \Cref{fig:subfig1}).

We show in \Cref{lem:relations_equivalent} that the equivalence relation of \Cref{def:AllermanRationalEquivalence} is equivalent to the existence of a tropical curve in $B \x \R$ that is `horizontal at infinity' and asymptotic to $V_{\pm}$ at $\pm \infty \in \R$ (see \Cref{def:horizontal_at_infinity} for a precise definition and \Cref{fig:subfig2} for a geometric intuition).
The $0$-th tropical Chow group $\CH_0(B)$ is then the free abelian groups on points of $B$ modulo relations induced by tropical curves in $B\x\R$ that are horizontal at infinity. as depicted in \Cref{fig:subfig2}.

\begin{remark}
    Although we restrict the ambient space to be a tropical affine manifold (instead of an arbitrary tropical cycle like in \cite{allermann2016rational}), note that the generators of $\CH_*(B)$ are in fact arbitrary tropical cycles in $B$ (not only integral affine submanifolds of $B$).
\end{remark}

\begin{figure}
    \centering
    \begin{subfigure}[b]{0.45\textwidth}
        \centering
        
        \begin{tikzpicture}[thick, scale=0.8]
\draw[step=1cm,gray,very thin] (-4,-2) grid (2,4);
            
            \draw (-4,0) -- (-1,0)  -- (0,1) -- (0,-2);
            \draw (-1,0) -- (-1,-2);
            \draw (0,1) -- (1,4);
            \node at (-2,0.5) {$1$};
            \node at (-1.5,-1) {$1$};
            \node at (0.5,-0.5) {$2$};
            \node at (-1,1) {$1$};
            \node at (0,2.5) {$1$};
        \end{tikzpicture}

        \caption{Tropical curve inside $\R^2$. The numbers next to the edges indicate the weight of each edge. At each vertex, the weighted sum of the outpointing primitive vectors vanishes.}
        \label{fig:subfig1}
    \end{subfigure}
    \hfill
    \begin{subfigure}[b]{0.45\textwidth}
        \centering
        \hspace{-1cm}
        \begin{tikzpicture}[thick,scale=0.6]

\begin{scope}[scale = 2]
                \draw (0,-1.5) node (v1) {} -- (0,0.5) node (v2) {} -- (1,1.5) node (v3) {} ;
                \draw[dashed]	(1,1.5) node (v3) {} -- (1,-0.5) node (v5) {}  -- (0,-1.5);
            \end{scope}
            
            \begin{scope}[shift={(7,0)},scale=2]
                \draw (0,-1.5) node (v1) {} -- (0,0.5) -- (1,1.5) node (v4) {} -- (1,-0.5) node (v6) {}  -- cycle;
            \end{scope}

\draw  (v2) rectangle (v1);
            \draw  (v3) edge (v4);
            \draw[dashed]  (v5) edge (v6);

\draw[->,>=stealth', line width=1pt] (0,-4) -- (8,-4);
            \node[scale = 1.5] at (8,-4.5) {$\mathbb R$};
            \node[scale = 1.5] at (-0.5,-0.5) {$B$};
            
\draw[color=red] (1,-0.5) -- (3.5,-0.5) -- (3.5,-2) -- (3.5,-0.5) -- (4,0) -- (3.5,0.5) -- (1,0.5) -- (3.5,0.5) -- (3.5,2);
            \draw[color=red] (4,0) -- (5,0)  -- (7,0.5)  -- (8,1.5);
            \draw[color=red] (5,0) -- (6,-0.5) -- (5,-1.5) --  (6,-0.5) -- (7.5,-0.5);
            \draw[color=red] (7,0.5)  -- (8.5,0.5);

            \node at (0.5,0.5) {$p_1^-$};
            \node at (0.5,-0.5) {$p_2^-$};
            \node at (8,-0.5) {$p_1^+$};
            \node at (8.5,1) {$p_2^+$};
        \end{tikzpicture}

        \caption{Tropical curve in $B \x \R$ which is `horizontal at infinity'. It gives the relation $p_1^+ + p_2^+ = p_1^- + p_2^-$ in $\CH_0(B)$. As drawn, it could give a relation for both $B = \R^2$ and $B = T^2$.}
        \label{fig:subfig2}
    \end{subfigure}
    \caption{Examples of tropical curves}
    \label{fig:mainfigure}
\end{figure}

 \subsection{Main results}
Let $T_\Z^*B$ be the local system of abelian groups whose sections over $U \subset B$ are differential $1$-forms $\alpha \in H^0(T^*U)$ such that 
$$
\alpha(v)\in \Z,\quad v \in T_\Z U.
$$
Note that such sections are locally constant in affine coordinates.
Let $\wedge^p T^*_\Z B$ be the exterior power of $T^*_\Z B$; global sections $\omega \in H^0(\wedge^p T^*_\Z B)$ are called \emph{tropical $p$-forms}.
The following is the main result of this paper:

\begin{mainthm}[= \Cref{thm:mainthm}]\label{mainthm:mainthm}
    If a compact integral affine manifold $B$ admits a non-zero tropical $p$-form for some $p \geq 2$, the group $\CH_0(B)_{\hom}$ is infinite dimensional (in the sense of \Cref{def:infinitedimensional}).
\end{mainthm}
Note that, according to \Cref{mainthm:mainthm}, the existence of a non-zero tropical $1$-form does not imply the tropical Chow group is infinite-dimensional.
This can already be seen in the $1$-dimensional case: the tropical Abel-Jacobi theorem \cite[Theorem 6.2]{mikhalkin2008tropical} implies that a tropical circle has
$$
\CH_*(S^1) \cong H_*(S^1) \oplus \Jac(S^1),
$$
which is finite-dimensional. 
Our second main result, \Cref{mainpr:mainpr}, shows that $H^0(T_\Z^*B) \neq 0$ is also not sufficient to guarantee infinite-dimensionality of $\CH_0(B)_{\hom}$ even when $\dim B >1$.

Let us first recall the definition of the \emph{tropical Albanese variety}, which will replace the role of $\Jac(B)$ in the tropical Abel-Jacobi theorem.
It is defined as the quotient 
$$
\Alb(B) = \frac{\Hom(H^0(T_\Z^*B),\R)}{H_1(B;\Z)}
$$
where $1$-cycles include into $\Hom(H^0(T_\Z^*B),\R)$ via integration.
The tropical Chow group of points of degree zero admits a map
\begin{align*}
    \alb:\CH_0(B)_{\hom} \to \Alb(B)\\
    b^+-b_-\mapsto \int_\gamma -,
\end{align*}
where $\gamma$ is any curve in $B$ with $\partial\gamma = b^+ - b^-$ (note that any two such $\gamma$ differ by a $1$-cycle, so the map is well-defined).

To show that the existence of tropical $1$-forms is also not sufficient to guarantee infinite-dimensionality of $\CH_0(B)_{\hom}$ it is enough to consider the lowest dimensional case  $\dim B = 2$.
By dimension reasons $H^0(\wedge^p T_\Z^*B) = 0$ for $p > 2$, whereas $H^0(\wedge^2 T_\Z^*B) = 0$ if and only if $B$ is non-orientable.
The only $2$-dimensional closed integral affine manifolds (in particular, no singularities in the integral affine structure)  which are non-orientable are \emph{tropical Klein bottles}.
These have been classified by Sepe \cite{sepe2010classification}.
We work with the following family of tropical Klein bottles (this is the family numbered by (3) in \cite{sepe2010classification} and called `Klein bottle of type I' in \cite{muniz2024lagrangian}).
Let $\Gamma = \vspan{a,b |aba = b}$ be the fundamental group of a Klein bottle.
We consider the injection
\begin{align}\label{eq:injectionKlein}
    \begin{split}
        \Gamma &\into \Aff_\Z (\R^2) := \R^2 \rtimes \GL(2,\Z)\\
        a &\mapsto ((x,y)\mapsto (x, y + y_0))\\
    b &\mapsto ((x,y)\mapsto (x + x_0, -y))
    \end{split}
\end{align}
where $x_0,y_0 \in \R$.
The quotient
$$
K := \R^2/\Gamma
$$
inherits an integral affine structure.
We show:
\begin{mainprop}[=\Cref{pr:albKleiniso}]\label{mainpr:mainpr}
    The tropical Albanese map
    $$
    \alb:\CH_0(K)_{\hom}\to \Alb(K)
    $$
    is an isomorphism. 
    In particular, $\CH_0(K)_{\hom}$ is finite-dimensional.
\end{mainprop}

\begin{rmk}
    Existence of non-zero tropical $1$-forms is instead related to \emph{non-discreteness} of the tropical Chow group: the split short exact sequence
    $$
    0 \to \CH_0(B)_{\hom} \to \CH_0(B) \to \Z \to 0
    $$
    together with the surjective Albanese map $\CH_0(B)_{\hom}  \to \Alb(B)$ show that if $\Alb(B)$ is non-trivial (which is equivalent to $H^0(T^*_\Z B) \neq 0$, since the $\R$-vector space $\Hom(H^0(T^*_\Z B), \R)$ is zero if and only if the free $\Z$-module $H^0(T^*_\Z B)$ is zero), then $\CH_0(B)_{\hom}$ cannot be discrete.
    In other words, not all cycles of the same degree are equivalent to one another.
\end{rmk}

 \subsection{Connections to other work}
As mentioned in \Cref{sec:Chowgroupsintro}, \Cref{mainthm:mainthm} can be thought of as a tropical version of Mumford's theorem \cite{mumford1969rational} when $\dim B = 2$ and Roitman's generalization \cite{roitman1971gamma} for arbitrary $\dim B \geq 2$.
In fact, Mumford and Roitman's results can be obtained as a corollary of \Cref{mainthm:mainthm} when $Y$ admits a tropicalization $B$, as long as one assumes that analytic curves tropicalise (see \Cref{rmk:tropicalisation}).

To make this more precise, let us recall from \cite{kontsevich2006affine} that associated to any integral affine manifold $B$ there is a rigid-analytic space $Y = Y(B)$. 
Rigid-analytic spaces are the analog of (complex) analytic spaces for non-archimedean fields.
In our case, $Y$ is an analytic space  over the Novikov field: this is the non-archimedean local field 
$$
\Lambda_{\Bbbk} = \left\lbrace \sum_{i=0}^\infty a_i T^{\lambda_i} \mid a_i \in \Bbbk, \lambda_i \in \R, \lim_{i\to \infty} \lambda_i = \infty \right\rbrace.
$$
The space $Y$ comes with a projection $\pi:Y \to B$.
We consider the  natural surjective map 
\begin{align}\label{eq:Z0YtoCHB}
    \begin{split}
        Z_0(Y) &\to \CH_0(B)\\
        [p] &\mapsto [\pi(p)].
    \end{split}
\end{align}
If $C \subset Y \x \P^1$ is a curve exhibiting some rational equivalence between points $(y_i^-)$ and $(y_i^+)$, under the assumption that analytic curves tropicalise the tropicalization $\pi(C)$ of $C$ is a tropical curve in $B \x \R$ exhibiting a tropical linear equivalence between the points $(\pi(y_i^-))$ and $(\pi(y_i^+))$ (in the sense of \Cref{def:tropical_linear_equivalence}).
It follows that the map in \Cref{eq:Z0YtoCHB} descends to a map
$$
    \CH_0(Y) \to \CH_0(B).
$$
Since this map is surjective, the commutative diagram
$$
\begin{tikzcd}
Y^k \x Y^k \ar[r] \ar[d, twoheadrightarrow] & \CH_0(Y)_{\hom} \ar[d, twoheadrightarrow] \\
B^k \x B^k \ar[r] & \CH_0(B)_{\hom}
\end{tikzcd}
$$
implies that if $\CH_0(B)_{\hom}$ is infinite-dimensional then so is $\CH_0(Y)_{\hom}$.
\begin{remark}
    Note that both Mumford and Roitman work with \emph{projective} varieties.
    If $Y = Z^{an}$ is the analytification of a projective variety $Z$, then by the rigid-analytic GAGA principle one has that $\CH_*(Y) \cong \CH_*(Z)$ (the former group being built from analytic subvarieties, the latter from algebraic subvarieties). 
    In this case we recover Mumford and Roitman's results.
    On the other hand, our result seems to be more general.
    For instance, when $Y$ is an (analytic) torus its tropicalisation is an integral affine torus. 
    In this case it is known that projectivity of $Y$ translates into $B$ being \emph{polarised} (this means it admits a strictly convex function with certain properties, cf. \cite{mikhalkin2008tropical}). 
    Our result works for general $B$ (not necessarily polarised) and would thus relate to the dimensionality of Chow groups of (non-algebraic) rigid-analytic varieties.
\end{remark}

\begin{remark}\label{rmk:tropicalisation}
    The fact that $\pi(C)$ is a tropical curve whenever $C \subset Y$ is an irreducible analytic curve seems to be well-known to experts but has only appeared in the literature under certain assumptions, see \cite{yu2015balancing,baker2016nonarchimedean}.
\end{remark}

\begin{remark}
    The existence of non-realizable tropical curves (that is, tropical curves in $B \x \R$ that do not admit a lift to $Y \x \P^1$) could a priori suggest that $\CH_0(B)$ should be `smaller' than $\CH_0(Y)$ (as every relation in $\CH_0(Y)$ gives one in $\CH_0(B)$, but not the other way around).
    However, \Cref{mainthm:mainthm} shows that, even if smaller, these non-realizable tropical curves are not enough to make $\CH_0(B)$ finite-dimensional.
\end{remark}

The results of Mumford were later translated into the world of symplectic geometry by Sheridan-Smith \cite{sheridan2020rational} and the corresponding generalization of Roitman by the present author in \cite[Proposition 1.6]{muniz2024lagrangian}, showing analogous statements for the so-called \emph{Lagrangian cobordism groups}.
All these works exploit a common idea (which goes back to the work of Severi): finding a dimension bound on a certain space by showing that it is isotropic for some differential form.
We emphasize that, while our proof is based on the same method, it is not a direct translation of either Mumford or Roitman's  nor Sheridan-Smith's proof. 
The former involve the study of differential forms on singular spaces and does not translate straightforwardly to tropical geometry, while the latter exploits a flexibility in Lagrangian cobordisms that is also not present in the tropical setting.
Instead, we use the rigidity of tropical geometry and in particular tropical curves to give a purely tropical proof of our result.

One can also think of \Cref{mainpr:mainpr} as a (partial) tropical version of  \cite{bloch1976zero}.
The authors show that if $Y$ is a smooth projective complex surface with Kodaira dimension $\kappa(Y) < 2$, then the vanishing of $H^0(\mcal K_Y)$ implies $\CH_0(Y)_{\hom} \cong \Alb(Y)$ (in particular, $\CH_0(Y)_{\hom}$ is finite-dimensional).
Bielliptic surfaces are examples of such surfaces, and they tropicalize to tropical Klein bottles; moreover, those of type I tropicalize to tropical Klein bottles of the type considered in this paper.
Hence \Cref{mainpr:mainpr} is the tropical version of the result in \cite{bloch1976zero} when $Y$ is an algebraic bielliptic surface of type I.

We lastly mention a connection between \Cref{mainpr:mainpr} and Lagrangian cobordism groups in symplectic geometry.
If $B$ is an integral affine manifold then one can consider the symplectic manifold $X(B):= T^*B / T^*_\Z B$ together with the Lagrangian torus fibration $X(B) \to B$.
Points in $B$ correspond to Lagrangian torus fibers in $X(B)$, and the work of Sheridan-Smith \cite{sheridan2021lagrangian} exhibits a relationship between (cylindrical) Lagrangian cobordisms between fibers of $X(B) \to B$ and tropical linear equivalence between points in $B$: showing that trivalent tropical curves admit Lagrangian lifts, they construct a map 
$$
\CH_0^{tri}(B) \to \Cob_{fib}(X(B))
$$
for any integral affine manifold $B$. 
(Here the superscript `tri' means a tropical Chow group whose relations come only from trivalent tropical curves in $B \x \R$, and the subscript 'fib' means a Lagrangian cobordism group generated only by fibers of the fibration $X(B) \to B$)
For $B = K$ a tropical Klein bottle of type I and $\mcal K:= X(K)$ the associated symplectic manifold, the subgroup 
$$
\Cob^{trop}(\mcal K) \subset \Cob(\mcal K)
$$
of the Lagrangian cobordism group of $\mcal K$ generated by tropical Lagrangians---that is, lifts of tropical subvarieties---was computed by the author in \cite{muniz2024lagrangian}.\footnote{\cite{muniz2024lagrangian} computes the \emph{planar} cobordism group of $\mcal K$, which is generally bigger than the cylindrical cobordism group. However, the other main theorem in that paper shows that the planar cobordism group is isomorphic to the Grothendieck group of the Fukaya category, and this isomorphism implies that the planar and cylindrical cobordism groups are actually isomorphic.}
The subgroup $\Cob_{fib}(\mcal K) \subset \Cob^{trop}(\mcal K)$ generated by fibers was found to be isomorphic to $\Z \oplus S^1$.\footnote{In \cite{muniz2024lagrangian} the Lagrangians and cobordisms are equipped with additional data, namely Pin and brane structures as well as a $G$-local system for a fixed abelian group $G$. The results in that paper compute the cobordism group including such data, hence the statement therein is $\Cob_{fib}(\mcal K) \cong \Z \oplus (S^1 \oplus G)$ (see \cite[Proposition 4.25]{muniz2024lagrangian}). If one forgets local systems (or equivalently, if one sets $G = 0$) then we obtain the above statement $\Cob_{fib}(\mcal K) \cong \Z \oplus S^1$.}
We thus see that the map 
\begin{align*}
    \CH_0(K)\cong\CH_0^{tri}(K) &\to \Cob_{fib}(\mcal K)\\
    [p] &\mapsto [F_p]
\end{align*}
is an isomorphism.
(The isomorphism $\CH_0(K)\cong\CH_0^{tri}(K)$ follows from noting that all the tropical curves in the proof of \Cref{mainpr:mainpr} are trivalent.)
In fact, in the present case we have
\begin{lem}\label{lem:CHtoCob}
    There is a well-defined map
    \begin{align*}
        \CH_*(K) &\to \Cob^{trop}(\mcal K)\\
        V & \mapsto L_V
    \end{align*}
    mapping a tropical subvariety $V \subset K$ to the corresponding Lagrangian lift $L_V \subset \mcal K$.
\end{lem}
\begin{proof}
    This map is well-defined on generators since (i) points (resp. the whole base) clearly lift to fibers (resp. the zero-section)  and (ii) tropical curves are in this case tropical hypersurfaces, whose lift exists by results of Hicks \cite{hicks2020tropical}.
    Furthermore, it also respects relations coming from tropical linear equivalence: for curves between points this follows from the results of \cite{sheridan2021lagrangian}, whereas tropical surfaces in $K \x \R$ between tropical curve in $K$ are tropical hypersurfaces and thus lift to Lagrangian surfaces in $\mcal K \x \C^*$ (again by \cite{hicks2020tropical}).
\end{proof}

The computations of the tropical homology of a Klein bottle in \cite{jell2018lefschetz} together with \Cref{mainpr:mainpr} show:
\begin{corollary}
    Let $K$ be a tropical Klein bottle of type I. 
    Then
    $$
    \CH_*(K) \cong (\Z \oplus \Alb(K))\oplus(\Z^2 \oplus \Z_2 \oplus \Pic^0(K))\oplus \Z
    $$
with $\Alb(K) \cong S^1 \cong \Pic^0(K)$.
\end{corollary}
\begin{proof}
    We consider the cycle-class map
    $$
        \CH_i(K) \to H_{i,i}(K)
    $$
    where $H_{i,j}(-)\equiv H_{i,j}(-;\Z)$ denote the tropical homology groups \cite{itenberg2019tropical}.
    These maps are clearly surjective for $i=0,2$.
    Furthermore, the tropical homology groups of $K$ have been computed for Klein bottles in \cite{jell2018lefschetz}, and it was shown there that $H_{2,0}(K;\R)=0$ and that this implies surjectivity for $i=1$.
    Thus the cycle-class map is surjective on every degree and we have a short exact sequence
    $$
        0 \to \CH_*(K)_{\hom} \to \CH_*(K) \to H_{*,*}(K) \to 0.
    $$
    We have that $\CH_2(K)_{\hom} = 0$ and $\CH_1(K)_{\hom} = \Pic^0(K)$.
    By \Cref{mainpr:mainpr} we also have $\CH_0(K)_{\hom} \cong \Alb(K)$.
    The result now follows by splitting the above short exact sequence  and using the computation $H_{1,1}(K) \cong \Z^2 \oplus \Z_2$ from \cite{jell2018lefschetz}.
\end{proof}

Together with the main theorem of \cite{muniz2024lagrangian} this shows:
\begin{prop}
    The map $\CH_*(K) \to \Cob^{trop}(\mcal K)$ of \Cref{lem:CHtoCob} is an isomorphism.
\end{prop}
\begin{remark}
    As pointed out in footnote 3, for this statement to hold the Lagrangians in the group $\Cob^{trop}(\mcal K)$ are not carrying a local system.
\end{remark}

 \paragraph{Acknowledgments} 
I would like to thank my advisor Nick Sheridan for his constant support and guidance.
I am grateful to Jeff Hicks for reading and commenting on an earlier draft of this paper, as well as pointing out that one can recover Mumford and Roitman's result when there is a tropicalization.
I also thank Arend Bayer for suggesting to try to find a tropical analog of Mumford's theorem, which led to the results in this paper.
This work was supported by an ERC Starting Grant (award number 850713-HMS). 
\section{Preliminaries}
Recall an \emph{integral affine manifold} is a smooth manifold $B$ together with an integrable lattice $T_\Z B \subset TB$ of tangent vectors.
This data is equivalent to an atlas of charts whose transition functions belong to the group $\Aff_\Z(\R^n)$ of integral affine transformations of $\R^n$, i.e. those of the form 
$$
\mathbf{x} \mapsto A\mathbf{x} + \mathbf{x}_0, A \in \GL(n,\Z), \mathbf{x}_0 \in \R^n.
$$
In this section, we introduce $1$-dimensional tropical subvarieties of $B$ (so-called \emph{tropical curves}), their moduli-spaces and deformations.
\subsection{Background on tropical curves}\label{sec:backgroundtropicalcurves}
We recall the notions of abstract tropical curves, tropical curves in a tropical manifold and parametrised tropical curves, following \cite{mikhalkin2005enumerative}.

Given a tropical manifold $B$, a \emph{tropical curve} in $B$ is a weighted,\footnote{In our definition, weights are non-negative.} balanced rational polyhedral complex of pure dimension one.
We now define an intrinsic notion of a tropical curve (that is, with no reference to an ambient manifold $B$) and then identify tropical curves in $B$ with morphisms from these `abstract' tropical curves to $B$.

Let $\Gamma$ be a graph, possibly with semi-infinite edges\footnote{Note that infinite edges, such as a graph with one edge and no vertices, are not allowed} and no $1$-valent vertices.
A \emph{tropical structure} on $\Gamma$ is the data of a tropical affine structure on each of its edges such that the induced Riemannian metric $dx^2$---where $dx \in T^*_\Z e$ is a generator---is complete.
A graph $\Gamma$ together with a tropical structure will be called an \emph{abstract tropical curve}.
\begin{remark}
    When making statements that hold for any abstract tropical curve with underlying graph $\Gamma$ (which we refer to as \emph{tropical curves whose topological type is that of $\Gamma$}) we will use the notation $\Gamma$, with no reference to the metric structure.
\end{remark}

\begin{defn}\label{def:parametrized_tropical_curve}
    Let $B$ be a (non-necessarily compact) integral affine manifold.
    A \emph{parametrized tropical curve} in $B$ is the data of an abstract  tropical curve $\Gamma$ together with a proper embedding $h: \Gamma \to B$ such that:
    \begin{itemize}
        \item its restriction to each edge is an integer affine map; that is, $dh(T_\Z e) \subset T_\Z B$ for all $e\in E$;
        \item at each vertex $v$ of $\Gamma$ the balancing condition
        \begin{equation}
            \sum_{e \to v} dh(z_e) = 0
        \end{equation} 
        holds, where $z_e \in T_{\Z,v} e$ is a primitive vector pointing outwards on the direction of $e$.
        The notation $e \to v$ in the summation means summing over all edges $e$ which have $v$ as one of their end vertices.
    \end{itemize}
    Two parametrized tropical curves $h_i: \Gamma_i \to B, i =1,2$ are \emph{equivalent} if there exists a isometry $f: \Gamma_1 \to \Gamma_2$ such that $h_1 = h_2 \circ f$.
\end{defn}

The image $h(\Gamma)$ of a parametrized tropical curve defines a tropical curve in $B$: the weight of  
(the image of) an edge $h(e)$ is the unique positive integer $w_e \in \N$ such that $dh(z_e) = w_e u_e$, where $z_e \in T_{\Z,v} e$ and $u_e \in T_{\Z,h(v)}B$ are both primitive vectors.
Conversely, every tropical curve (that is, every weighted, balanced rational polyhedral complex of pure dimension one) can be parametrized in the sense of \Cref{def:parametrized_tropical_curve}, and such parametrization is unique.
We will therefore use the terms tropical curve and parametrized tropical curve interchangeably.

\begin{remark}
    Our definition of parametrized tropical curve is slightly different from that in \cite[Definition 2.2]{mikhalkin2005enumerative}, where the map $h$ is not required to be an embedding and edges can be contracted to points. 
    Clearly every parametrized tropical curve in our sense is a parametrized tropical curve in the sense of \cite{mikhalkin2005enumerative}.
    Conversely, given a parametrized tropical curve in the sense of \cite{mikhalkin2005enumerative}, there always exists (possibly after changing the underlying abstract tropical curve) a parametrized tropical curve in our sense with the same image in $B$.
    Since the end goal is to have a parametrized description of tropical curves in $B$, both definitions are equivalent in this sense.
    The main advantage of our definition is the following: given any tropical curve $C\subset B$, there exists a \emph{unique} parametrized tropical curve (up to equivalence) with image $C$.
\end{remark}

We also introduce the following notation, which will be useful in later sections.
Let $E = E(\Gamma)$ be the set of edges of $\Gamma$ and $E_\infty \subset E$ the set of semi-infinite edges.
We denote by $V = V(\Gamma)$ the set of vertices of $\Gamma$ and by $\partial\Gamma$ the points at infinity of the unbounded edges of $\Gamma$.\footnote{That is, $\partial\Gamma= \bar{\Gamma}\setminus\Gamma$, where $\bar{\Gamma}$ is the compactification of $\Gamma$.}
Choosing an orientation for each edge we get maps 
\begin{equation}\label{eq:edge_maps}
    v^\pm : E \to V \sqcup \partial\Gamma
\end{equation}
that assign to each edge its initial and final vertex.

 \subsection{Moduli of parametrized tropical curves}\label{sec:moduli}
We denote by $\mcal M_\Gamma^{imm}(B)$  the moduli-space of equivalence classes of immersed parametrized tropical curves in $B$ whose topological type is that of $\Gamma$ (that is, an element in $\mcal M_\Gamma^{imm}(B)$ is a map $h: \Gamma \to B$ satisfying the same conditions as in \Cref{def:parametrized_tropical_curve} but which is only a proper immersion).
Choosing an orientation on each edge gives a preferred generator $u_e \in T_e\Gamma$ for each $e \in E$, which in turn yields an embedding
    \begin{align}\label{eq:embeddingmodulispace}
        \begin{split}
        \mcal M^{imm}_\Gamma(B) & \into B^{V}\x (T_\Z B\x (0,+\infty])^{E}\\
        (h:\Gamma\to B) & \mapsto ((h(v))_{v\in V},(dh(u_e),len(h(e)))_{e\in E}).
        \end{split}
    \end{align}
Here $len(h(e)) \in (0,+\infty]$ is the length of the interval $h(e)$.

\begin{lem}\label{lem:modulispaceisaffine}
The above map gives $\mcal M^{imm}_\Gamma(B)$ the structure of a tropical affine manifold.
\end{lem}
\begin{proof}
    Recall that a choice of orientation on the edges also gives maps $v^\pm : E \to V \cup \partial \Gamma$ recording the vertices of an edge $e$ (see \Cref{eq:edge_maps}).
    Given a point
    $$
    (\mathbf{q},\mathbf{f},\mathbf{l}) = ((q_v)_{v \in V},(f_e,l_e)_{e \in E}) \in B^{V}\x (T_\Z B\x (0,+\infty])^{E},
    $$
    we will use the notation $q_e^\pm = q_{v^\pm(e)}$ to denote the values of $\mathbf q$ at the positions corresponding to the vertices of $e$.
    Also, for any point $q \in B$ and any tangent vector $f \in T_{\Z,q}B$, we will denote by 
    \begin{align*}
        \gamma_{f,q}:\R \to  B\\
        \gamma_{f,q}(t) = q + t f
    \end{align*} 
    the curve starting at $q$ and with integral tangent vector $f$ (we parallel transport using the local system $T_\Z B$).
    Then the map above identifies $\mcal M^{imm}_\Gamma(B)$ with (an open subset of) the set of points $ (\mathbf{q},\mathbf{f},\mathbf{l})$ such that:
    \begin{enumerate}
        \item $q^+_e = \gamma_{f_e,q^-_e}(l_e/w_e)$;
        \item $\sum_{e\to v} \pm w_e f_e=0$, where the $\pm$ sign depends on whether $e$ has been oriented as to point away from or into $v$.
\end{enumerate}
    That is, the first condition ensures that the set of points $\mathbf{q}$ can be joined by lines of integer slope (the slopes given by $\mathbf{f}$) and prescribed length, while the second guarantees that the balancing condition is met at each vertex, thus giving a tropical curve. Since the tangent space to (the image of) $\mcal M_\Gamma^{imm}(B)$ is an integer affine subspace of $TB^V$, the result follows.
\end{proof}

\begin{remark}
    The moduli-space of tropical curves and its structure of a tropical affine manifold have appeared in the literature before, see e.g. \cite{gathmann2009tropical} for the case $B = \R^n$.
    We work with a different (local) embedding (compare our \Cref{eq:embeddingmodulispace} with \cite[Proposition 4.7]{gathmann2009tropical}) since it is more convenient for our purposes to remember the images of all the vertices of $\Gamma$ (see e.g. the proof of \Cref{prop:Ziscountableunion}).
\end{remark}

We now consider the restriction of the map (\ref{eq:embeddingmodulispace}) to $\mcal M_\Gamma(B) \subset \mcal M^{imm}_\Gamma(B)$.
Note this is now an embedding.
\begin{cor}\label{cor:modulispaceisaffine}
    The moduli-space $\mcal M_\Gamma(B)$ admits a tropical affine structure.
\end{cor}
\begin{proof}
    This follows from \Cref{lem:modulispaceisaffine} together with the fact that $\mcal M_\Gamma(B) \subset \mcal M^{imm}_\Gamma(B)$ is open.
\end{proof}

 \subsection{Deformations of tropical curves}\label{sec:deformations}
Let $\Gamma$ be an abstract tropical curve.
\begin{defn}\label{def:cotangent_sheaf}
    The \emph{sheaf of locally constant $1$-forms on $\Gamma$} is the sheaf
    $$
        \mcal T^*(U)=\left\lbrace\begin{array}{ll}
        H^0(T^*_\Z e)\otimes\R, & U\subset e\\
        \ker(f:\displaystyle\bigoplus_{e\to v} T_v^*e\to \R), & U \tx{ neighborhood of a vertex } v
        \end{array} \right.,
    $$
    where $f(\alpha_1,\dots,\alpha_k)=\sum_i \alpha_i(u_{e_i})$, $u_{e_i}$ being a primitive vector pointing outwards on the direction of $e_i$.
\end{defn} 
\begin{remark}
    Note that, over an open set $U \subset e$, we have that $H^0(T^*_\Z e)\otimes\R \cong \R$, an isomorphism given by $(u, t) \mapsto tu$.
\end{remark}
\begin{example}
    Consider the curve $\Gamma$ in \Cref{fig:sheaf}.
    Then any tuple of the form 
    $$
        (adx + bdy, -bdy, -adx) \in \bigoplus_{e_i \to v} T^*_v e_i, \quad a, b \in \R
    $$
    is a section of $\mcal T$ over $U$, since $(adx + bdy)(u_1) -bdy(u_2) - adx(u_3) = 0$.
    Further, the restriction of $\mcal T^*$ to an open set $U_j \subset U$ contained in the $j$-th vertex is just the projection $\oplus_{e_i \to v} T^*_v e_j \to T^*_v e_j \cong H^0(T^*_{\Z, v} e_j) \otimes \R$.
\end{example}

\begin{figure}
    \centering
    \begin{tikzpicture}[thick, scale=0.8, shift={(2,0)}]
\draw[step=1cm,gray,very thin] (-4,-2) grid (2,4);
        
        \draw[red] (-4,1) -- (-1,1)  -- (2,4);
        \draw[red] (-1,1) -- (-1,-2);
        \node at (-1.2, 0.8) {$v$};

        \draw[dashed] (-1,1) circle (2);
        \node at (-3.3, 0.5) {$U$};

\draw[->] (-1, 1) -- (0, 2);
        \node at (-0.6, 1.7) {$u_1$};

        \draw[->] (-1, 1) -- (-1, 0);
        \node at (-0.5, 0.5) {$u_2$};

        \draw[->] (-1, 1) -- (-2, 1);
        \node at (-1.5, 1.3) {$u_3$};

    \end{tikzpicture}
\caption{A tropical curve (in red) together with a collection of primitive integral vectors, one for each edge, pointing outwards.}
\label{fig:sheaf}
\end{figure}

The following is a relative version of \cite[Lemma 3.5]{sheridan2021lagrangian} and \cite[Lemma 6.1]{mikhalkin2008tropical}.
We use the notation introduced at the end of  \Cref{sec:backgroundtropicalcurves}.
\begin{lem}\label{lem:1formsandhomology}
    There is an isomorphism $H^0(\mcal T^*)\cong H_1(\bar{\Gamma},\partial\Gamma;\R)$.
\end{lem}
\begin{proof}
    Consider the linear map
    \begin{align*}
        g:\R^E & \to \R^{V\sqcup\partial\Gamma} \to \R^V
    \end{align*}
    given by composing $1_e \mapsto 1_{v^+(e)} - 1_{v^-(e)}$ with the projection $\R^{V\sqcup\partial\Gamma} \to \R^V$, where $1_e$ is the generator of the $\R$-factor corresponding to $e$, $1_v$ is the generator of the $\R$-factor corresponding to $v$ and $v^\pm: E \to V\sqcup \partial\Gamma$ are the maps defined in \Cref{eq:edge_maps}.
    First note that the group $H_1(\bar{\Gamma},\partial\Gamma;\R)$ is canonically identified with the kernel of $g$ using simplicial homology.
    On the other hand, a choice of orientation for the edges of $\Gamma$ gives preferred generators $u_e \in T_\Z e$ and hence preferred isomorphisms $H^0(T^*_\Z e) \otimes \R \cong \R$.
    It is then clear that an element of $\oplus_e H^0(T^*_\Z e) \otimes \R$ extends to a global section of $\mcal T^*$---that is, it is in the kernel of $f$---if and only if the corresponding element in $\R^E$ lies in the kernel of $g$. \end{proof}

\begin{rmk}\label{rmk:explicitiso}
    Using a simplicial model for $H_1(\bar{\Gamma},\partial\Gamma;\R)$, an explicit isomorphism $\eta:\mcal T^*(\Gamma) \xra{\sim}H_1(\bar{\Gamma},\partial\Gamma;\R)$  is given by sending a $1$-form $\alpha\in T^*e$ to the $1$-chain $\alpha(u_e)e \in C_1^{simp}(\bar{\Gamma};\R)$, where $u_e$ is a primitive vector in $T_\Z e$ whose direction is determined by the orientation of $e$.
\end{rmk}

Now let $h:\Gamma\to B$ be a parametrized tropical curve in $B$; for simplicity, we will refer to it as $\Gamma$, even though the map $h$ is part of the definition.
\begin{defn}    
    The \emph{sheaf of deformations of $\Gamma$} is the sheaf
    $$
        \mcal D(U)=\left\lbrace
        \begin{array}{ll}
        T_p B/T_p h(e), & p \in U\subset e\\
        T_vB, & U \tx{ neighborhood of the vertex }v
        \end{array} 
        \right..
    $$
    Note that there are canonical isomorphisms $T_p B/T_p h(e) \cong T_q B/T_q h(e)$ for any $p,q \in e$ given by parallel transport, hence the sheaf is well-defined.
    The restriction maps are given by the projections $T_v B\to T_vB/T_v e \cong T_pB/T_p e$, the latter isomorphism given again by parallel transport. 
    A \emph{deformation} of $\Gamma$ is a section of $\mcal D$, and we write $\deff(\Gamma) := H^0(\mcal D)$ for the group of (global) deformations.
\end{defn}

\begin{example}
    Consider again the tropical curve in \Cref{fig:sheaf}. 
    The vector $v_1 \in T^*_v B$ is a section of $\mcal D(U)$.
    The restriction to an open set $U_j \subset U$ contained in the $j$-th edge is just its equivalence class $[v_1] \in T_v B / T_v e_j$, which can be thought of as the projection of $v_1$ to an orthogonal complement to $e_j$.
    The same applies to $v_2$ and $v_3$.
\end{example}

\begin{lem}\label{lem:deformationsaretangentspace}
    There is an identification $T_h \mcal M_\Gamma(B) \cong \deff(\Gamma)$.
\end{lem}
\begin{proof}
    A deformation of $\Gamma$---that is, a global section of $\mcal D$---is determined by its stalks on the vertices; that is, by an element 
    $$
    (u_v)_{v\in V} \in \oplus_v T_v B.
    $$
    Such a tuple determines a deformation if and only if $[u_{v^+(e)}] = [u_{v^-(e)}] \in TB/Te$ for every edge $e$.

    Let us now show how to view $T_h\mcal M_\Gamma(B)$ as the same subset of $\oplus_v T_v B$.
    Recall from the proof of \Cref{lem:modulispaceisaffine} that the embedding (\ref{eq:embeddingmodulispace}) identifies $\mcal M_\Gamma(B)$ with (an open subset of) the set of points
    $$
    (\mathbf{q},\mathbf{f},\mathbf{l}) = ((q_v)_{v \in V},(f_e,l_e)_{e \in E}) \in B^{V}\x (T_\Z B\x (0,+\infty])^{E},
    $$
    such that:
    \begin{enumerate}
        \item $q^+_e = \gamma_{f_e,q^-_e}(l_e/w_e)$;
        \item $\sum_{e\to v} \pm w_e f_e=0$, where the $\pm$ sign depends on whether $e$ has been oriented as to point away from or into $v$.
\end{enumerate}
    Here, $q_e^\pm$ are the ends of an edge $e \in E$ and $\gamma_{f_e,q_e^-}: \R \to B$ is the curve whose tangent vector is (the parallel transport of) $f_e \in T_{\Z,q_e^-} B$ and such that $\gamma_{f_e,q_e^-}(0) = q_e^-$.
    In particular, note that the projection 
    $$
    \mcal M_\Gamma(B) \into B^{V}\x (T_\Z B\x (0,+\infty])^{E} \to B^V
    $$
    is a local embedding (once the position of the vertices is known, the curve is fixed).
    It follows that we can identify $T_h\mcal M_\Gamma(B)$ with a subset of $\oplus_{v} T_{h(v)} B$.
    Lastly, given $(u_v)_{v \in V} \in \oplus_{v} T_{h(v)} B$, condition 1. above is equivalent to $[u_{v^+(e)}] = [u_{v^-(e)}] \in TB/Te$ for every edge $e$.
    Indeed, once we perturb $v^-(e)$ via $u_{v^-(e)}$, the fact that $v^\pm(e)$ are connected by a curve of direction $f_e$ implies that $u_{v^+(e)} = u_{v^-(e)} + \lambda f_e$, where $\lambda = 0$ if and only if the perturbed edge keeps the same length (see \Cref{fig:perturbation}).
\end{proof}

\begin{figure}
    \centering
         \begin{tikzpicture}[thick] 
            \begin{scope}[scale=0.8]
\begin{scope}
                    \draw[red] (-4,0) -- (-1,0)  -- (0,1) -- (0,-2);
                    \draw[red] (-1,0) -- (-1,-2);
                    \draw[red] (0,1) -- (1,4);
                \end{scope}

                \begin{scope}[shift={(-0.4, 0.7)}]
                    \draw[red, dashed] (-4,0) -- (-1,0)  -- (0,1) -- (0,-2);
                    \draw[red, dashed] (-1,0) -- (-1,-2);
                    \draw[red, dashed] (0,1) -- (1,4);
                \end{scope}

                \draw[->] (-1, 0) -- (-1.4, 0.7);
                \node at (-1.8, 1) {$u_{v^-(e)}$};
                \draw[->] (0, 1) -- (-0.4, 1.7);
                \node at (-1.2, 1.7) {$u_{v^+(e)}$};
            \end{scope}

            \begin{scope}[scale=0.8, shift={(8, 0)}]
\begin{scope}
                    \draw[red] (-4,0) -- (-1,0)  -- (0,1) -- (0,-2);
                    \draw[red] (-1,0) -- (-1,-2);
                    \draw[red] (0,1) -- (1,4);
                \end{scope}

                    \draw[red, dashed] (-4,0) -- (-1,0)  -- (0.5, 1.5) -- (0.5,-2);
                    \draw[red, dashed] (-1,0) -- (-1,-2);
                    \draw[red, dashed] (0.5,1.5) -- (1.5, 4.5);

                    \node at (-1.7, 0.6) {$u_{v^-(e)} = 0$};
                    \draw[->] (0, 1) -- (0.5, 1.5);
                    \node at (1.2, 1.4) {$u_{v^+(e)}$};
            \end{scope}
        \end{tikzpicture}
    
    \caption{Two types of deformations of a tropical curve. On the left, we have deformations where $u_{v^-(e)} = u_{v^+(e)} \in TB$, i.e. \emph{translations} of the curve. On the right we have deformations where $u_{v^-(e)}=0$ and $u_{v^+(e)} \in Te$ (in particular, $[u_{v^-(e)}] = [u_{v^+(e)}] = 0 \in TB/Te$). These deformations \emph{expand} (or contract) one edge. Similarly, one could have deformations with $u_{v^-(e)}\in Te$ and $u_{v^+(e)} = 0$.}
    \label{fig:perturbation}
\end{figure}

We have defined two sheaves: the sheaf $\mcal T^*$ of locally constant $1$-forms and the sheaf $\mcal D$ of deformations.
The following Lemma relates the two sheaves:

\begin{lem}\label{lem:deformationsand1forms}
    For every locally constant $(k+1)$-form $\omega \in H^0(\wedge^{k+1}T^*B)$, the map
    \begin{align*}
        \wedge^k T_v B &\to \oplus_{e \to v} T_v^*e\\
        u& \mapsto (w_e\iota_u \omega)_{e\to v}
    \end{align*}
    extends to a sheaf homomorphism
    $$
    \Phi_\omega: \wedge^{k}\mcal D \to \mcal T^*.
    $$
    Here we denote by $\iota_u \omega = \omega(u\wedge-)$ the contraction of $\omega$ with $u$.
\end{lem}   
\begin{proof}
    The map is clearly compatible with restriction.
    Hence it is enough to show that $(w_e\iota_u \omega)_{e\to v} \in \ker(f)$ (cf. \Cref{def:cotangent_sheaf}).
    This follows from the balancing condition: if $u = u_1\wedge\dots\wedge u_k \in \wedge^kT_vB$, then one has 
    $$
    f((w_e\iota_u \omega)_{e\to v}) = \sum_{e\to v} w_e \omega(u\wedge u_{e}) =   \omega(u\wedge \sum_{e\to v} w_e u_{e}) = 0
    $$ since $\sum_{e\to v} w_e u_{e} = 0$.
\end{proof}

Putting together Lemmas \ref{lem:1formsandhomology} and \ref{lem:deformationsand1forms}, we obtain a map
$$
\wedge^k \tx{def}(\Gamma) = \wedge^k H^0(\mcal D) \to H^0(\mcal T^*) \cong H_1(\bar{\Gamma},\partial\Gamma;\R).
$$

 \subsection{Tropical linear equivalence}\label{sec:tropical_linear_equivalence}
We now introduce a particular type of tropical curves in $B \x \R$; these will be used to define the equivalence relation of tropical linear equivalence (see \Cref{def:tropical_linear_equivalence}).
From now on we assume that $B$ is compact.

\begin{defn}\label{def:horizontal_at_infinity}
    We say that a parametrized tropical curve $h:\Gamma \to B \x \R$ is \emph{horizontal at infinity} if the composition $pr_B \circ h\restr{e}$ is a constant map for every semi-infinite edge $e \in E_\infty$.
\end{defn}

We denote by $\mcal M^{hor}_\Gamma (B \x \R) \subset \mcal M_\Gamma (B \x \R)$ the moduli-space of parametrized tropical curves $h: \Gamma \to B \x \R$ which are horizontal at infinity (that is, those tropical curves in $B \x \R$ that are horizontal at infinity and whose topological type is that of $\Gamma$).
Then it follows from  \Cref{def:horizontal_at_infinity} that there are  well-defined maps
\begin{align}\label{eq:evaluation_at_infinity}
    \begin{split}
        ev_{\pm\infty}: \mcal M^{hor}_\Gamma (B \x \R) &\to \prod_{e\in E_{\pm\infty}} B^{w_e}\\
        h &\mapsto ((pr_B \circ h(e))^{w_e})_{e \in E_{\pm\infty}}
    \end{split}
\end{align}
that record the images, with multiplicity, of the semi-infinite edges of $\Gamma$ asymptotic to $\pm\infty$.
Here, we used the notation $b^n := (b,\dots,b) \in B^n$ for $b \in B$ and $n \in \N$, and we decompose $E_\infty = E_{-\infty} \cup E_{+\infty}$ depending on whether (the image of) an edge is asymptotic to $B \x \{\pm \infty\}$.

Note that the balancing condition implies that $\sum_{e \in E_{\pm\infty}} \pm w_e = 0$ (see \cite[Proposition 3.2, (e)]{allermann2016rational} for a proof, together with \Cref{lem:relations_equivalent}), hence the maps $ev_{-\infty}$ and $ev_\infty$ have as target a common $B^n$.
Composing $(ev_{-\infty},ev_{+\infty})$ with the map
\begin{align}\label{eq:projection_to_Z0hom}
    \begin{split}
        B^n \x B^n &\to  Z_0(B)\\
        (\mathbf{b}^+,\mathbf{b}^-) &\mapsto \sum_i (b_i^+ - b_i^-)
    \end{split}
\end{align}
yields a map 
$$
\mcal M^{hor}_\Gamma (B \x \R) \to Z_0(B).
$$
Let $R_0(B) \subset Z_0(B)$ be the subgroup generated by the image of $\cup_\Gamma \mcal M^{hor}_\Gamma(B \x \R) \to Z_0(B)$.

\begin{defn}\label{def:tropical_linear_equivalence}
    We say $0$-cycles $V_\pm \in Z_0(B)$ are \emph{(tropically) linearly equivalent}  if $V_+ - V_- \in R_0(B)$.\footnote{Roughly speaking, one can think of two cycles $V_\pm$ being tropically linearly equivalent if there exists a tropical curve $h:\Gamma \to B \x \R$ which coincides with $V_\pm \x \R$ on a neighborhood of $B \x \{\pm\infty\}$. This is however not precise: for instance, the cycles $V_\pm$ could be tropically linearly equivalent because there is a tropical curve $h:\Gamma \to B \x \R$ which coincides with $(V_\pm \cup W)\x \R$ on a neighborhood of $B \x \{\pm\infty\}$ for some set of points $W \subset B$, but a curve coinciding with $V_\pm \x \R$ need not exist.}
    The \emph{tropical $0$-th Chow group} $\CH_0(B)$ is the quotient 
    $$
    \CH_0(B) =Z_0(B)/R_0(B).
    $$
\end{defn}

Although this definition seems different from \Cref{def:AllermanRationalEquivalence}, we have:
\begin{lem}\label{lem:relations_equivalent}
    Let $\widetilde{\CH_0}(B)$ be the Chow group obtained with the definition of rational equivalence in \Cref{def:AllermanRationalEquivalence}.
    Then there is an isomorphism $\widetilde{\CH_0}(B) \cong\CH_0(B)$.
\end{lem}
\begin{proof}
    The proof is identical to that in \cite[Proposition 3.5]{allermann2016rational} and uses the `full graph' construction of \cite{mikhalkin2006tropical} (also named `tropical graph cobordism' in \cite{allermann2010first}), see \Cref{fig:graph}
\end{proof}

In the remaining of this paper we will use \Cref{def:tropical_linear_equivalence}, as it is more convenient for our proofs.

\begin{figure}
    \begin{tikzpicture}[scale=1, line cap=round, line join=round, >=stealth]

\draw[thick] (0,0) -- (3.5,0);
    \node at (0, -0.5) {$B$};

    \draw[thick, ->] (-0.5, 0.5) -- (-0.5, 1.5);
    \node at (-0.5, 1.7) {$\R$};
    
    \draw[thick, red] (0,1) -- (1,1) -- (1.5, 1.5) -- (2, 1.5) -- (2.5, 1.0) -- (3.5,1.0);
    \node[red] at (1.75, 1.75) {$\text{graph}(f)\subset B \times \R$};

\draw[thick,->] (4,0.9) -- (7.5,0.9);

\draw[thick, red] (8,1) -- (9,1) -- (9.5, 1.5) -- (10, 1.5) -- (10.5, 1.0) -- (11.5,1.0);
    \draw[thick, red] (9, 1) -- (9, -2);
    \draw[thick, red] (9.5, 1.5) -- (9.5, 4);
    \draw[thick, red] (10, 1.5) -- (10, 4);
    \draw[thick, red] (10.5, 1) -- (10.5, -2);

    \node[red] at (7.5, 3) {$\text{full graph}(f) \subset B \times \R$};

    \end{tikzpicture}

    \label{fig:graph}
    \caption{Schematic picture of full graph construction. Given a function $f : B \to \R$ (whose graph we represent in the left), the full graph construction extends its bending locus towards $\pm\infty\in \R$, creating a curve that is horizontal at infinity.}
\end{figure} 
\section{Tropical forms and dimensionality of tropical Chow groups}\label{sec:formsanddimension}
\subsection{Infinite dimensionality}
Let $\Psi: B^{2n} \to \CH_0(B)$ be the composition of the map in \Cref{eq:projection_to_Z0hom} with the projection $Z_0(B) \to \CH_0(B)$; that is, $\Psi$ is the map defined by
$$
\Psi(\mathbf{b}^+,\mathbf{b}^-) = \sum_{i=1}^n ([b^+_i] - [b^-_i]),
$$
where $\mathbf{b}^\pm = (b^\pm_1,\dots,b^\pm_n) \in B^n$.
We consider the subset ${\mcal Z_{n,k}}\subset B^{2n}\x B^{2k}$  defined as
$$
{\mcal Z_{n,k}}:=\set{(\mathbf{b}^+,\mathbf{b}^-,\mathbf{c}^+,\mathbf{c}^-) \st \Psi(\mathbf{b}^+,\mathbf{b}^-)=\Psi(\mathbf{c}^+,\mathbf{c}^-)}.
$$

To ellucidate the structure of $\mcal Z_{n,k}$, we will need the following two Lemmas:

\begin{lem}\label{lem:reductiontosinglecurve}
    If $(\mathbf{b}^+,\mathbf{b}^-,\mathbf{c}^+,\mathbf{c}^-) \in {\mcal Z_{n,k}}$, then there exist $m \in \N$, $\mathbf{d} \in B^{m}$ and a parametrized tropical curve $h:\Gamma \to B \x \R$ such that 
    \begin{align}\label{eq:extra_points_at_infinity}
        \begin{split}
            (ev_{-\infty}, ev_{\infty})(\Gamma) 
            = 
            (\mathbf{b}^- \cup \mathbf{c}^- \cup \mathbf{d}, \mathbf{b}^+  \cup  \mathbf{c}^+ \cup \mathbf{d})
        \in
        B^{n+k+m} \x B^{n+k+m}.
        \end{split}
    \end{align}
    Here the maps $ev_{\pm\infty}$ are those defined in \Cref{eq:evaluation_at_infinity}.
\end{lem}
\begin{proof}
    Let $(\mathbf{b}^+,\mathbf{b}^-,\mathbf{c}^+,\mathbf{c}^-) \in {\mcal Z_{n,k}}$.
    This means there exist tropical curves $h_i : \Gamma_i \to B \x \R$ which are horizontal at infinity and such that
    \begin{equation*}
        \prod_i (ev_{-\infty}, ev_{\infty})(\Gamma_{i}) 
        = 
        (\mathbf{b}^- \cup \mathbf{c}^- \cup \mathbf{d}, \mathbf{b}^+  \cup  \mathbf{c}^+ \cup \mathbf{d})
        \subset
        B^{n+k+m} \x B^{n+k+m}.
    \end{equation*}
    for some $\mathbf{d}\in B^{m}$.
    Indeed, note that $(\mathbf{b}^+,\mathbf{b}^-,\mathbf{c}^+,\mathbf{c}^-) \in {\mcal Z_{n,k}}$ means that $\sum_i b^+_i + \sum_i c^+_i \sim  \sum_i b^-_i + \sum_i c^-_i$ in $\CH_0(B)$, and relations in $\CH_0(B)$ are, by definition, generated tropical curves that are horizontal at infinity.
    Note that the union $\cup_i h_i(\Gamma_i)$ is a tropical curve in $B \x \R$.
    We argued in \Cref{sec:backgroundtropicalcurves} that every tropical curve admits a parametrization: hence there exists a parametrized tropical curve $h: \Gamma \to B \x \R$ with $h(\Gamma) = \cup_i h_i(\Gamma_i)$.
\end{proof}

\begin{lem}\label{lem:reduction_to_curves}
    Let $\Gamma$ be an abstract tropical curve. 
    Then the image of the evaluation map  
    $$
        (ev_{-\infty}, ev_{\infty}): \mcal M^{hor}_\Gamma(B \x \R) \to B^{d} \x B^{d}
    $$ 
    is a countable union of open subsets of affine subspaces.
    \end{lem}
    \begin{proof}
        Recall from \Cref{cor:modulispaceisaffine} that $\mcal M^{hor}_\Gamma(B\x\R) \subset \mcal M^{hor,imm}_\Gamma(B\x\R)$ can be realized as (an open subset of) an integral affine submanifold of $(B\x \R)^V \x (T_\Z(B\x \R)\x(0,+\infty])^E$ via the map in \Cref{eq:embeddingmodulispace}.\footnote{Strictly speaking we have only proved this for $\mcal M_\Gamma(B\x\R) \subset \mcal M_\Gamma^{imm}(B \x \R)$. However, note that $\mcal M^{hor}_\Gamma(B\x\R) \subset \mcal M_\Gamma(B\x\R)$ is the intersection of $\mcal M_\Gamma(B\x\R)$ with the submanifold of $(B\x \R)^V \x (T_\Z(B\x\R) \x (0,+\infty])^E$ whose projection to $T_\Z(B\x \R)^{E_\infty}$ lies in $T_\Z B \subset T_\Z(B\x \R)$. This shows that $\mcal M^{hor}_\Gamma(B\x\R)$ is still an integral affine submanifold of $(B\x \R)^V \x (T_\Z(B\x\R) \x (0,+\infty])^E$.}
        This map gives $\mcal M^{hor}_\Gamma(B\x\R)$ its tropical affine structure, and the map $(ev_{-\infty}, ev_{\infty}): \mcal M^{hor}_\Gamma(B \x \R) \to B^{d} \x B^{d}$ is identified with the projection 
        $$
        (B\x\R)^{V}\x (T_\Z(B\x \R)\x (0,+\infty])^{E} \to (B\x\R)^{V_{\infty}},
        $$
        where  $V_{\infty}$ is the set of vertices from which the unbounded edges emanate.
        Since projecting is an affine map, and $\mcal M_\Gamma^{hor}(B\x\R)$ is an affine subspace, the result follows.
    \end{proof}

We can now show:
\begin{prop}\label{prop:Ziscountableunion}
    The set ${\mcal Z_{n,k}}$ is a countable union of open subsets of affine subspaces of $B^{2n} \x B^{2k}$.
\end{prop}
\begin{proof}
    For each $m \in \N$, we consider the space 
    $$
        \mcal M^{hor}(B\x\R)_m = \bigcup_{|E(\Gamma)_\infty| = 2n+2k+2m} \mcal M^{hor}_\Gamma(B \x \R),
    $$
    where the union ranges over all topological types with $2n+2k+2m$ semi-infinite edges.
    Let 
    $$
    \mcal M^{hor}(B\x\R)_m^{\mcal Z_{n,k}} \subset \mcal M^{hor}(B\x\R)_m
    $$
    be the subset of curves $\Gamma \in \mcal M^{hor}(B\x\R)_m$ that satisfy \Cref{eq:extra_points_at_infinity} (that is, they have $n + k + m$ ends at each $\pm \infty$, and $m$ of the asymptotic points coincide on both ends).
    Note that $\mcal M^{hor}(B\x\R)_m^{\mcal Z_{n,k}}$ admit maps $\mcal M^{hor}(B\x\R)^{\mcal Z_{n,k}}_m \to {\mcal Z_{n,k}}$ by composing
    \begin{equation}\label{eq:longmap}
        \mcal M^{hor}(B\x\R)^{\mcal Z_{n,k}}_m \xra{(ev_{-\infty}, ev_{\infty})} B^{n+k+m} \x B^{n+k+m} \to \mcal Z_{n,k} \subset B^{2n} \x B^{2k},
    \end{equation} 
    where the second map is the product of the projections $B^{n+m+k} \to B^{n+m}$ (with a final 
    reordering of factors $B^{n+m} \times B^{n+m} \to B^{2n} \to B^{2m}$).
    Furthermore, \Cref{lem:reductiontosinglecurve} shows that the union of these maps 
    $$
    \bigcup_m  \mcal M^{hor}(B\x\R)^{\mcal Z_{n,k}}_m \to \mcal Z_{n,k} \subset B^{2n} \x B^{2k}.
    $$
    surjects onto $\mcal Z_{n,k}$.
    Note that this is a \emph{countable} union, hence it is enough to show that the images of the maps $\mcal M^{hor}(B\x\R)^{\mcal Z_{n,k}}_m \to B^{2n} \x B^{2k}$ satisfy the statement of the Proposition.
    Since the second map in \Cref{eq:longmap} is a projection, which is affine, we can further reduce to showing that the image of the evaluation maps
    $$
        (ev_{-\infty}, ev_{\infty}): \mcal M^{hor}(B \x \R)_m^{\mcal Z_{n,k}} \to B^{n+m+k} \x B^{n+m+k}
    $$
    satisfy the statement of the Proposition.

    First note that this is true for the extended map 
    $$
    (ev_{-\infty},ev_\infty):\mcal M^{hor}(B \x \R)_m \to B^{n+m+k} \x B^{n+m+k}
    $$
    (note we now consider $\mcal M^{hor}(B \x \R)_m$ instead of $\mcal M^{hor}(B \x \R)_m^{\mcal Z_{n,k}}$ as the domain). 
    Indeed, the moduli-space $\mcal M^{hor}(B\x \R)_m$ and the above map are defined as a \emph{countable} union, and each piece 
    $$
    \mcal M_\Gamma^{hor}(B \x \R)  \to B^{n+m+k} \x B^{n+m+k}
    $$
    satisfies the statement of the Proposition by \Cref{lem:reduction_to_curves}.

    To conclude, note that under the embedding 
    $$
    \mcal M^{hor}_\Gamma(B \x \R)_m \into (B\x\R)^{V}\x (T_\Z(B\x \R)\x (0,+\infty])^{E}
    $$ 
    the subspace $\mcal M^{hor}_\Gamma(B \x \R)_m^{\mcal Z_{n,k}}$ is identified with those tuples in the image whose projections to the last $m$ factors in $B^{V_{+\infty}}$ and $B^{V_{-\infty}}$ agree (possibly up to permutation).
    These are integral linear conditions and hence exhibit $\mcal M^{hor}_\Gamma(B \x \R)_m^{\mcal Z_{n,k}} \subset \mcal M^{hor}_\Gamma(B \x \R)_m$ as a finite union of integral affine subspaces, from which the result follows.

\end{proof}

\begin{rmk}\label{rmk:Zivsimageevaluation}
Denote by $\mcal Z_{n,k} = \cup_i Z_i$ the decomposition given by \Cref{prop:Ziscountableunion}. 
Note that each $Z_i$ is given as (a linear subspace of) the  image of  evaluation maps $(ev_{-\infty},ev_{+\infty}):\mcal M_\Gamma^{hor}(B\x\R) \to B^{n+m+k} \x B^{n+m+k}$.
Therefore, in what follows we will reduce statements about the $Z_i$ to the images of evaluation maps.
\end{rmk}

We now show that the (open subsets of) affine subspaces $Z_i$ are isotropic for a suitable family of  $p$-forms on $B^{2n} \x B^{2k}$.
To simplify the notation, we will write $p_i : B^{2n} \x B^{2k} \to B$ for the composition of the permutation 
$$
(1,\dots,2n+2k) \mapsto (1,\dots,n,2n+1,\dots,2n+k,n+1,\dots,2n,2n+k+1,\dots,2n+2k)
$$
with the projection onto the $i$-th factor.

\begin{prop}\label{prop:Zisotropic}
Let $0\neq \tilde\omega \in H^0(\wedge^p T^*_\Z  B)$ be a non-zero tropical $p$-form.
Then the $p$-form
\begin{equation}\label{eq:omega}
\omega=\sum_{i=1}^{n+k}p_i^*\tilde\omega - \sum_{i=n+k +1}^{2n+2k}p_i^*\tilde \omega \in H^0(\wedge^pT^*_\Z(B^{2n}\x B^{2k}))
\end{equation}
vanishes on the $Z_i$.
\end{prop} 
\begin{proof}
    By \Cref{rmk:Zivsimageevaluation} it is enough to show that $\omega$ vanishes on the image of the evaluation maps $(ev_{-\infty},ev_{+\infty}):\mcal M_\Gamma^{hor}(B\x\R) \to B^{2n} \x B^{2k}$.
    Let then $h:\Gamma \to B\x \R$ be a parametrized tropical curve that is horizontal at infinity.
    Let $\eta: H^0(\mcal T^*) \xra{\sim} H_1(\bar\Gamma,\partial\Gamma;\R)$ be the isomorphism from \Cref{lem:1formsandhomology}.
    Denote by $\mcal T^*_\infty \Gamma$ the direct sum of the stalks of $\mcal T^*$ at the unbounded edges of $\Gamma$.
    Then we have a commutative square
    $$
    \begin{tikzcd}
    H^0(\mcal T^*)  \arrow[r,"\sim","\eta"'] \arrow[d,"(-)|_{E_\infty}"] & H_1(\bar\Gamma,\partial \Gamma;\R) \arrow[d,"\partial"]\\
    \mcal T^*_\infty \Gamma \arrow[r,"\sim"] &H_0(\partial \Gamma;\R).
    \end{tikzcd}
    $$

    Recall from \Cref{sec:deformations} that the isomorphism $\eta:H^0(\mcal T^*) \xra{\sim} H_1(\bar{\Gamma},\partial\Gamma;\R)$ depends on a choice of orientation for the edges.
    For the purposes of this proof, we choose an orientation such that the unbounded edges---which are horizontal at infinity by \Cref{def:horizontal_at_infinity}---agree with $\partial_t$, where $t$ is the coordinate of $\R$.
    Then in the square above the map $\eta$ sends 
    $$
    (\alpha_e)\in \oplus_e T^*e \mapsto \sum_e \alpha_e(u_e)e \in H_1(\bar\Gamma,\partial\Gamma;\R)
    $$
    (see \Cref{rmk:explicitiso}); the bottom horizontal map is the restriction  
    $$
    (\alpha_e)_{e\in E_\infty}\mapsto (\pm\alpha_e(\partial_t))_{e\in E_{\pm\infty}};
    $$
    the left vertical map is restriction to the stalks; and the right vertical map is the boundary operator of the long exact sequence of the pair $(\bar\Gamma,\partial\Gamma)$, which sends a simplicial $1$-chain to its boundary.
    As before we denote by $\mcal D_\infty\Gamma \simeq \oplus_{e\in E_\infty}\mcal D_e$ the direct sum of the stalks of $\mcal D$ at the unbounded edges of $\Gamma$; note that $\mcal D_\infty\Gamma \simeq TB^{2n + 2k}, 2n + 2k = |E_\infty|$.
    Then using \Cref{lem:deformationsand1forms} we can extend the above square to a commutative diagram
    \begin{equation}\label{eq:commutativediagram}
        \begin{tikzcd}
            \wedge^{p}H^0(\mcal D) \arrow[r,"\Phi_{\tilde\omega\wdg dt}"] \arrow[d,"(-)|_{E_\infty}"'] & H^0(\mcal T^*)  \arrow[r,"\sim","\eta"'] \arrow[d,"(-)|_{E_\infty}"] & H_1(\bar{\Gamma},\partial \Gamma;\R) \arrow[d,"\partial"]\\
            \wedge^p \mcal D_\infty \Gamma   \arrow[r,"\Phi_{\tilde\omega\wdg dt}"] &
            \mcal T^*_\infty \Gamma \arrow[r,"\sim"] \arrow[dr,"ev"'] &H_0(\partial \Gamma;\R)\arrow[d,"\iota_*"]\\
            & & H_0(\bar{\Gamma};\R)\cong \R
            \end{tikzcd}
    \end{equation}
    where $ev = ev_{(1,\dots,1,-1,\dots,-1)}$ maps $(\alpha_e)_{e\in E_\infty}$ to $\sum_{e\in E_{\pm\infty}} \pm\alpha_e(\partial_t)$.
    Furthermore, for $v_i = (v_i^1,\dots,v_i^{2n+2k}) \in T(B^{2n} \x B^{2k}), i = 1,\dots,p$, we have
    \begin{align*}
    \omega(v_1\wdg\dots\wdg v_p) 
    &=
    \sum_{i=1}^{2n}\tilde\omega(v^i_1\wdg\dots\wdg v_p^i) - \sum_{i=2n+1}^{2n+2k}\tilde\omega (v^i_1\wdg\dots\wdg v_p^i) \\
    &=
    \sum_{i=1}^{2n}\Phi_{\tilde\omega \wdg dt}(v^i_1\wdg\dots\wdg v_p^i)(\partial_t) + \sum_{i=2n+1}^{2n+2k}\Phi_{\tilde \omega \wdg dt}(v^i_1\wdg\dots\wdg v_p^i)(-\partial_t)\\
    & = (ev\circ \Phi_{\tilde\omega\wdg dt})(v_1\wdg\dots\wdg v_p)
    \end{align*}
    where in the second equality we have used the identification $TB \cong \mcal D_e$ for any unbounded edge $e \in E_\infty$.
    The result now follows: if $v_1,\dots,v_p$ are tangent vectors to the image of an evaluation map 
    $$
    (ev_{-\infty},ev_\infty): \mcal M_\Gamma^{hor}(B\x\R) \to B^{n+m+k} \x B^{n+m+k},
    $$ 
    then  there exist deformations $D_i \in H^0(\mcal D)$ such that  $v_i= D_i|_{E_\infty}$, and
    \begin{align*}
    \omega(v_1\wdg\dots\wdg v_p)&=(ev \circ \Phi_{\omega\wdg dt})(D_1|_{E_\infty}\wdg \dots\wdg D_p|_{E_\infty})\\
    &=(ev \circ \Phi_{\omega\wdg dt})(D_1\wdg\dots \wdg D_p)|_{E_\infty}\\
    &= (\iota_*\circ\partial \circ \eta \circ \Phi_{\omega\wdg dt}) (D_1\wdg\dots \wdg D_p)\\
    &=0.
    \end{align*}
    Here we have used the commutativity of the top left square of \Cref{eq:commutativediagram} in the second line, the commutativity of both the top right square and the bottom right triangle in the third line, and the fact that  $\iota_*\circ\partial=0$ in the last line.
\end{proof}

Recall from \Cref{def:infinitedimensional} that we say $\CH_0(B)_{hom}$ is infinite-dimensional if the maps
    \begin{align}\label{eq:projection_to_CH0hom}
        \begin{split}
            \Psi:B^n\x B^n &\to \CH_0(B)_{\hom}\\ 
            (\mathbf{b}^+,\mathbf{b}^-) &\mapsto \sum_{i=1}^n \left( b_i^+ - b_i^-\right)
        \end{split}
     \end{align}
are not surjective for any $n$.
\Cref{prop:Zisotropic} together with the following Lemma will allow us to conclude that the existence of non-zero tropical $p$-forms, $p\geq 2$, implies infinite dimensionality of $\CH_0(B)_{\hom}$.

\begin{lem}\label{lem:roitmanbound}(\cite[Lemma 9]{roitman1971gamma})
    Let $V = \oplus_{j=1}^m V_j$ be a graded vector space and $0 \neq \omega_j \in \wedge^p V_j^*$ non-zero $p$-forms on $V_j$, $p \geq 2$.
    Denote by $pr_j: V \to V_j$  the natural projection.
    If $W \subset V$ is an isotropic subspace for the $p$-form 
    $$
    \Omega = \sum_j pr^*_j \omega_j,
    $$
    then $\dim W \leq \dim V - m$.
\end{lem}

\begin{thm}\label{thm:mainthm}
If a tropical affine manifold $B$ admits a non-zero tropical $p$-form for some $p \geq 2$, the group $\CH_0(B)_{\hom}$ is infinite-dimensional.
\end{thm} 

\begin{proof}
    Let $n\in \N$ arbitrary; we will show the map $\Psi:B^n\x B^n \to \CH_0(B)_{\hom}$ of \Cref{eq:projection_to_CH0hom} is not surjective.
    Note that $\Psi$ is surjective if and only if the projection
    $$
    \mcal Z_{n,k} \subset B^{2n} \x B^{2k} \to B^{2k}
    $$
    is surjective for all $k \in \N$.
    Let $0\neq \tilde \omega \in H^0(\wedge^p_\Z T^*B)$ be a non-zero tropical $p$-form and let $\omega$ be the form of \Cref{eq:omega}.
    It follows from \Cref{prop:Ziscountableunion} that $\mcal Z_{n,k} = \cup_i Z_i$ decomposes as a countable union.
    Furthermore, each $Z_i$ is isotropic for $\omega$ by \Cref{prop:Zisotropic}, thus
    $$
    \dim Z_i \leq \dim (B^{2n}\x B^{2k}) - (2n + 2k)
    $$ 
    by \Cref{lem:roitmanbound}.
    We have that $\mcal Z_{n,k} = \cup_i Z_i$ is a countable union of submanifolds of bounded dimension, hence the projection $\mcal Z_{n,k} \to B^{2k}$ is surjective if and only if at least one of the restrictions $Z_i \to B^{2k}$ is (the image of each $Z_i$ is an integral affine submanifold of $B^{2k}$, so if each such image is strictly lower dimensional, a countable union of them is also strictly lower dimensional).
    Now choose $k\in \N$ such that $\dim B^{2n} - 2n < 2k$, i.e. $\dim B^{2n} - 2n - 2k < 0$.
    Then we have
    \begin{align*}
         \dim Z_i &\leq \dim (B^{2n}\x B^{2k}) - (2n + 2k) \\
        & = \dim B^{2k} + (\dim B^{2n} - 2n - 2k)\\
        & < \dim B^{2k}
    \end{align*}
    so the projection cannot be surjective.
\end{proof}

\subsection{Tropical Chow group of the Klein bottle}\label{sec:chowklein}
In this section we consider a certain family of tropical Klein bottles and show their Chow group of points are finite-dimensional.
This shows that the existence of a non-zero tropical $1$-form does not imply the tropical Chow group is infinite-dimensional.

Let $\Gamma = \vspan{a,b |aba = b}$ be the fundamental group of a Klein bottle, and consider the injection $\Gamma \into \Aff_\Z (\R^2)$ given by \Cref{eq:injectionKlein}.
Let $K$ be the quotient
$$
K := \R^2/\Gamma
$$
with the induced integral affine structure.

\begin{figure}
  \centering
  \begin{tikzpicture}[thick]
    \begin{scope}[]
\node at (-4.5,-0.5) {$K$};

    \begin{scope}[shift={(-0.5,-10)}]
    \draw  (-3.5,9) rectangle (0.5,5);
    \node[scale = 1.5] at (-1.5,9) {$>>$};
    \node[scale = 1.5] at (-1.5,5) {$>>$};
    \node[rotate = 90,scale = 1.5] at (-3.5,7) {$>$};
    \node[rotate = 90,scale = 1.5] at (0.5,7) {$<$};
    \end{scope}
\begin{scope}[shift={(-0.5,-10)}]
    \filldraw[red]  (-2.5,7.5)  circle (2pt) node [anchor=south west] {$p$};
    \filldraw[red]  (-2.5,6.5)  circle (2pt) node [anchor=north west] {$\iota(p)$};
    \end{scope}
    \end{scope}

\draw[blue] (-4,-2.5) -- (0,-2.5);
    \draw[blue] (-4,-3.5) -- (0,-3.5);

    \node[blue] at (-0.5,-2) {$F^2_y$};

\begin{scope}[shift={(7,0)}]
    \node at (-4.5,-0.5) {$K$};
\begin{scope}[shift={(-0.5,-10)}]
    \draw  (-3.5,9) rectangle (0.5,5);
    \node[scale = 1.5] at (-1.5,9) {$>>$};
    \node[scale = 1.5] at (-1.5,5) {$>>$};
    \node[rotate = 90,scale = 1.5] at (-3.5,7) {$>$};
    \node[rotate = 90,scale = 1.5] at (0.5,7) {$<$};
    \end{scope}
\begin{scope}[shift={(-0.5,-10)}]
    \filldraw[red]  (-2.5,7.5)  circle (2pt) node [anchor=south west] {$p$};
    \filldraw[red]  (-2.5,6.5)  circle (2pt) node [anchor=north west] {$\iota(p)$};
    \end{scope}
    \end{scope}

    \draw[blue] (4,-1) -- (4,-5);
    \filldraw[red]  (4,-3)  circle (2pt) node [anchor=west] {$s(\theta)$};

    \node[blue] at (4.5,-1.5) {$F^1_\theta$};
  \end{tikzpicture}
  \caption{Left: geometry of the relation $2(p-\iota(p))\sim 0$. Right: geometry of the relation $4p - 4s(\theta) \sim 0$.}
  \label{fig:kleinrelations}  
\end{figure}

\begin{prop}\label{pr:albKleiniso}
The tropical Albanese map
$$
\alb:\CH_0(K)_{\hom}\to \Alb(K)
$$
is an isomorphism. 
In particular, $\CH_0(K)_{\hom}$ is finite-dimensional.
\end{prop} 
\begin{proof}
    Write $p_1: K \to B_1$ (resp. $p_2: K \to B_2$) for the projection onto the first (resp. second) coordinate.
    Note that $B_1 = \R/(x\sim x+ x_0) \simeq S^1$ and
    $$
      B_2 = \frac{\R}{y\sim y+ y_0,y \sim -y} \simeq [0,y_0/2].
    $$
    We will show that the tropical section
    \begin{align*}
      s:B_1 \simeq S^1 & \into K\\
    \theta & \mapsto (\theta,0)  
    \end{align*}
    of $p_1$ induces (by push forward) an isomorphism on $\CH_0$. 
    
    Since $s$ is a section of $p_1$, functoriality of Chow groups implies injectivity. 
    To prove surjectivity, we first show that for any point $p \in K$ one has
    $$
    2(p - \iota(p)) \sim 0
    $$
    in $\CH(K)$, where $\iota: K \to K$ is a reflection along the $x$-axis.
    Indeed, $p - \iota(p)$ lies in a fiber $F^2_y := p_2\inv(y) \simeq S^1$ of the projection $p_2$; note that $F^2_y$ is a tropical curve in $K$ (see the left part of \Cref{fig:kleinrelations}).
    It is easy to check that $2(p-\iota(p))$ is in the kernel of $\CH(F^2_y)_{\hom} \to \Alb(F^2_y)$, and since this map is an isomorphism by the tropical Abel theorem \cite[Theorem 6.2]{mikhalkin2008tropical} we conclude $2(p-\iota(p))\sim 0$ in $\CH(F^2_y)_{\hom}$. Pushing forward this relation under the inclusion $F^2_y \into K$ one obtains the result.

    Now let $p \in K$ be any point and denote $\theta = p_1(p)$.
    We consider the homologically trivial $0$-cycle
    $$
    4p - 4s(\theta) \sim 2p + 2\iota(p) - 4s(\theta), 
    $$
    where we have used the previous relation $2(p - \iota(p))\sim 0$.
    By definition of $\iota$ one has that 
    $$
    p + \iota(p) - 2s(\theta) = (p - s(\theta)) - (s(\theta) - \iota(p))
    $$
    lies in the kernel of $\CH_0(F^1_\theta)_{\hom} \to \Alb(F^1_\theta)$ (see right part of \Cref{fig:kleinrelations}), hence so does $4p - 4s(\theta)$.
    By the tropical Abel-Jacobi theorem, $4p \sim 4s(\theta)$ in $\CH_0(F^1_\theta)_{\hom}$, and  divisibility of $\CH_0(F^1_\theta)_{\hom} \simeq S^1$ further implies $p \sim s(\theta)$.
    Pushing forward this relation under the inclusion $F^1_\theta \into K$ yields the same result in $\CH_0(K)_{\hom}$, thus proving surjectivity.
\end{proof}

\begin{cor}
Let $K$ be a tropical Klein bottle. 
Then $\CH_0(K)=\Z\oplus S^1$
\end{cor}
\begin{proof}
Using the split short exact sequence
$$
0 \to \CH_0(K)_{hom} \to \CH_0(K) \to \Z \to 0
$$
and \Cref{pr:albKleiniso}, the statement of the Corollary is equivalent to $\Alb(K)=S^1$.
This is an immediate consequence of $H^0(T_\Z^* K)=\Z$, generated by the invariant covector of the monodromy.
\end{proof}

\renewbibmacro{in:}{}
\def\bibrangedash{ -- }
\printbibliography

\end{document}